\documentclass{amsart}

\usepackage{amsmath}
\usepackage{amsthm}
\usepackage{amsfonts}
\usepackage{amssymb}
\usepackage{amsbsy}
\usepackage{amscd}
\usepackage{eucal}

\newcommand{\beq}{\begin{equation}}
\newcommand{\eeq}{\end{equation}}

\newtheorem{theorem}{Theorem}[section]
\newtheorem{corollary}[theorem]{Corollary}
\newtheorem{lemma}[theorem]{Lemma}

\theoremstyle{definition}
\newtheorem{definition}[theorem]{Definition}
\theoremstyle{remark}
\newtheorem{remark}[theorem]{Remark}
\newtheorem*{example}{Example}

\numberwithin{equation}{section}

\newcommand{\rmd}{\mathrm{d}}
\newcommand{\rmi}{\mathrm{i}}

\newcommand{\I}{\mathbb{I}}

\newcommand{\cI}{\mathcal{I}}
\newcommand{\cE}{\mathcal{E}}
\newcommand{\cU}{\mathcal{U}}
\newcommand{\cH}{\mathcal{H}}
\newcommand{\cZ}{\mathcal{Z}}
\newcommand{\cS}{\mathcal{S}}

\newcommand{\fM}{\mathfrak{M}}

\newcommand{\ud}{\frac{1}{2}}
\newcommand{\uds}{\textstyle\frac{1}{2}}

\newcommand{\og}{{\mathop{\Bar{g}}}}

\newcommand{\sku}{\vspace*{0.1cm}}
\newcommand{\skd}{\vspace*{0.2cm}}
\newcommand{\skt}{\vspace*{0.3cm}}

\def\staccrel#1#2{\mathrel{\mathop{#1}\limits_{#2}}}

\begin{document}

\title[Fredholm integral equations of the first kind]
{Fredholm integral equations of the first kind and topological information theory}

\author[E. De Micheli]{Enrico De Micheli}
\address{IBF - Consiglio Nazionale delle Ricerche, Via De Marini, 6 - 16149 Genova, Italy}
\email{enrico.demicheli@cnr.it}

\author[G. A. Viano]{Giovanni Alberto Viano}
\address{Dipartimento di Fisica, Universit\`a di Genova -
Istituto Nazionale di Fisica Nucleare - Sezione di Genova,
Via Dodecaneso, 33 - 16146 Genova, Italy}



\begin{abstract}
The Fredholm integral equations of the first kind are a classical example
of ill-posed problem in the sense of Hadamard. If the integral operator
is self-adjoint and admits a set of eigenfunctions, then a formal solution can be written in terms of eigenfunction
expansions. One of the possible methods of regularization consists in truncating
this formal expansion after restricting the class of admissible solutions
through a-priori global bounds. In this paper we reconsider various
possible methods of truncation from the viewpoint of the $\varepsilon$-coverings
of compact sets. 
\end{abstract}

\maketitle

\section{Introduction}
\label{se:introduction}
We consider the Fredholm integral equations of the first kind
\beq
(Af)(x)=\int_a^b K(x,y)f(y)\,\rmd y = g(x) 
\qquad (a\leqslant x\leqslant b),
\label{1}
\eeq
whose kernel is supposed to be Hermitean and square-integrable, i.e.,
\beq
K(x,y) = \overline{K(y,x)}, \nonumber
\label{2}
\eeq
and
\beq
\int_a^b\left[\int_a^b\left|K(x,y)\right|^2\,\rmd x\right]\,\rmd y < \infty. \nonumber
\label{3}
\eeq
For simplicity, we shall suppose hereafter that the kernel $K$, the data function $g$ and the
unknown function $f$ are real-valued functions; in addition, we assume that the interval
$[a,b]$ is a bounded and closed subset of the real axis.
The operator $A$ acts as follows, $A:X\to Y$, where $X$ and $Y$ are respectively the solution and the data space.
We assume that $X\equiv Y\equiv L^2(a,b)$. Then $A$ is a self-adjoint compact operator.
Further, we assume throughout the paper that the range of $A$ is infinite dimensional. Accordingly,
the integral operator $A$ admits a set of eigenfunctions $\{\psi_k\}_{k=1}^\infty$ and, correspondingly, a countable infinite
set of eigenvalues $\{\lambda_k\}_{k=1}^\infty$. The eigenfunctions form an orthonormal
basis of the orthogonal complement of the null space of the operator $A$ and therefore
an orthonormal basis of $L^2(a,b)$ when $A$ is injective. Then the Hilbert-Schmidt theorem guarantees that 
$\lim_{k\to+\infty}\lambda_k=0$. Next we assume that the sequence of eigenvalues (which are supposed to be positive)
is (non-strictly) decreasing, counting multiple eigenvalues with respect to their multiplicity,
i.e., $\lambda_1\geqslant\lambda_2\geqslant\lambda_3\geqslant\cdots$. Let us however observe that
the assumption of positivity of the eigenvalues (here made for the sake of simplicity) can be
easily relaxed by considering in the subsequent analysis the moduli of the eigenvalues.

We suppose that the unique solution of the equation $Af=0$ is $f \equiv 0$, so that the uniqueness 
of the solution to \eqref{1} is guaranteed. But, as is well-known, uniqueness does not imply
(in the case considered here of $L^2$-spaces) continuous dependence of the solution on the data.
Next, by the Hilbert-Schmidt theorem we associate with the integral equation \eqref{1} the following
eigenfunction expansion:
\beq
f(x) = \sum_{k=1}^\infty \frac{g_k}{\lambda_k}\,\psi_k(x)
\qquad (x\in[a,b]),
\label{4}
\eeq
where $g_k=(g,\psi_k)$ ($(\cdot,\cdot)$ denoting the scalar product in $L^2(a,b)$). 
The series \eqref{4} converges in the $L^2$-norm.

The solution to Eq. \eqref{1} is however not so simple as one could expect
just looking at expansion \eqref{4}. The difficulties emerge in view of the
following problems.

\begin{itemize}
\item[(a)] The range of $A$ is not necessarily closed in the data space $Y$. Therefore,
given an arbitrary function $g\in Y$, there does not necessarily exists a solution $f\in X$.
\item[(b)] Even if two data functions $g_1$ and $g_2$ do belong to the range of $A$,
and their distance in $Y$ is small, nevertheless the distance between $A^{-1}g_1$
and $A^{-1}g_2$ can be arbitrarily large, in view of the fact that the inverse of
the compact operator $A$ is not bounded.
\end{itemize}

The difficulties mentioned above represent indeed the ill-posed character,
in the sense of Hadamard \cite{Hadamard}, of the Fredholm integral equations of the 
first kind (see, e.g., Ref. \cite{Wing}).

Let us now note that, in practice, there always exists some inherent noise affecting the data
(at least the roundoff numerical error) and, therefore, instead of Eq. \eqref{1}
we have to deal with the following equation (assuming an additive model of noise \cite{Bissantz,Eggermont}):
\beq
Af+n = \og \qquad (\og=g+n),
\label{5}
\eeq
where $n$ represents the noise. Therefore, instead of expansion \eqref{4}, we have to handle
the following expansion:
\beq
\sum_{k=1}^\infty \frac{\og_k}{\lambda_k}\,\psi_k(x),
\label{6}
\eeq
where $\og_k=(\og,\psi_k)$. Then the difficulties indicated in the points (a) and (b)
emerge clearly. We are thus forced to make use of the so-called methods of regularization.

The literature on these methods is very extensive and any list of references can hardly be exhaustive
(see, e.g., Refs. \cite{Eggermont,Engl,Groetsch,Louis,Mathe,Tenorio,Tikhonov} and the references quoted therein). 
In this paper we limit ourselves to consider only one of the possible approaches to regularization,
precisely, the procedure which consists in truncating suitably expansion \eqref{6}, that is,
stopping the summation at a certain finite value of $k$.
The simplest example of truncation is to stop expansion \eqref{6} at the largest value
of $k$ such that $\lambda_k\geqslant(\varepsilon/E)$ (where $\varepsilon$ is a bound on the norm of the noise
and $E$ is an a-priori global bound on the norm of the solution). This value of $k$ will be called the
truncation point and will be denoted hereafter by $k_1$. Of course, later in the paper (see Sect. \ref{se:entropy}),
we shall consider even other different types of truncation.

The main purpose of this work is connecting the truncation method of expansion \eqref{6} 
to the covering of \emph{compacta} in a sense that will be specified later in the paper
(see Definition \ref{def:1}, Remark \ref{rem:2}, and Lemma \ref{lem:2} for the case of compact ellipsoids).
In fact, the problem being considered can be reduced to the analysis of coverings of compact ellipsoids belonging
to the range of the operator $A$. This covering problem can be appropriately treated
within the framework of Kolmogorov's theory of $\varepsilon$-entropy and $\varepsilon$-capacity 
\cite{Kolmogorov,Lorentz}.
It is well-known that this theory, which makes use of general ideas of information and communication theory, 
is not founded on probabilistic methods. Accordingly, it can be properly called \emph{Topological Information Theory}.
Let us recall that the Kolmogorov theory of $\varepsilon$-entropy and $\varepsilon$-capacity
of compacta in functional spaces has played a relevant role in modern analysis, including
the problem of the representation of continuous functions of several variables by functions of one variable,
which is connected with Hilbert's Thirteenth Problem; the latter contained (implicitly) the conjecture
that not all continuous functions of three variables are representable as superpositions of continuous
functions of two variables. This Hilbert's conjecture was refuted in 1957 by Kolmogorov and Arnold
(see Refs. \cite[p. 169]{Lorentz} and \cite{Kolmogorov}). 
In the present paper we apply this theory to the regularization of Fredholm integral equations of first kind, 
obtained by means of truncation procedures. We thus complete a preliminary approach 
to this question given in Ref. \cite{DeMicheli}.

\skd 

The regularization obtained by truncating expansion \eqref{6} requires rather
restrictive assumptions in order to be numerically realizable:

\sku

\noindent
{\bf Assumption (A).} We assume that the significant contribution to the unknown
function $f$ (see \eqref{1}) is brought by those components $\og_k$ which are
retained by the appropriate truncation of expansion \eqref{6}, i.e., the spectral distribution $f_k$
($f_k=(f,\psi_k)$) of the function $f$ is assumed to be positively skewed so that neglecting the subset 
$\{\og_k\}_{k=k_1+1}^\infty$ of the data (if the truncation point is the value $k_1$ introduced above)
is indeed feasible.

\sku

Stated in other words, assumption (A) amounts to excluding those functions whose Fourier
components are small, or even null, for small values of $k$, while the significant contributions are
brought by the components at \emph{intermediate or high values} of $k$, which are
cut away by the truncation procedure. We shall return on this point with more details later in the paper.

Now, in order to establish the truncation point of expansion \eqref{6} by means of the theory
of covering of compacta, we need to make another assumption:

\sku

\noindent
{\bf Assumption (B).} The perturbation due to the noise must be such that the noisy data function $\og$
still belongs to the range of the operator $A$, i.e., $\sum_{k=1}^\infty (\og_k/\lambda_k) < \infty$.

\skd

The paper is organized as follows. In Sect. \ref{se:basic} we give the basic definitions
of the \emph{topological information theory}, and obtain a relevant inequality relating
$\varepsilon$-entropy and $\varepsilon$-capacity. In Sect. \ref{se:entropy}
we show how the regularization of Fredholm integral equations of the first kind,
obtained through truncation methods, can be reconsidered in the framework of the theory
of the covering of compacta. In the same section we establish relevant
inequalities for the $\varepsilon$-capacity and fix the truncation points
in expansion \eqref{6} in some significant cases. In Sect. \ref{se:strong}
we prove that the truncations obtained in Sect. \ref{se:entropy} lead to
regularized approximations. Section \ref{se:stability} is devoted to the analysis
of the stability estimates and of the related type of continuity (H\"{o}lder or logarithmic)
in the dependence of the solution on the data. In the same section
two remarkable examples are discussed. 
Finally, in Sect. \ref{se:conclusions} some conclusions are drawn.
 
\section{Basic definitions of topological information theory}
\label{se:basic}

Let $\cI$ be a nonvoid set in a metric space $Y$. We introduce the following definitions which have
been stated by Kolmogorov and Tihomirov \cite{Kolmogorov} (see also the book
of Lorentz \cite[Chapter 10]{Lorentz}).

\begin{definition}
A system $\gamma$ of sets $u_k\in Y$ is called an $\varepsilon$-\emph{covering} of the set $\cI$
if the diameter $d(u_k)$ of an arbitrary $u_k\in\gamma$ does not exceed $2\varepsilon$ and
$\cI\subseteq\cup_{u_k\in\gamma} u_k$.
\label{def:1} 
\end{definition}

\begin{definition}
A set $u\subseteq Y$ is called an $\varepsilon$-\emph{net} of the set $\cI$ if every point of the set $\cI$
is at a distance not exceeding $\varepsilon$ from some point of $u$.
\label{def:2}
\end{definition}

\begin{definition}
A set $u\subset Y$ is called $\varepsilon$-\emph{separated} if every pair of distinct points of $u$
are at a distance greater than $\varepsilon$ from each other.
\label{def:3}
\end{definition}

\noindent
Definition \ref{def:3} can be equivalently expressed as follows:
the points $x_1,\ldots,x_m$ of $\cI$ are called $\varepsilon$-\emph{distinguishable} if the
distance $\rho$ between each two of them exceeds $\varepsilon$, i.e.,
$\rho(x_i,x_k)>\varepsilon$ for all $i\neq k$.

\sku

\begin{remark}
(see \cite[Chapter 10]{Lorentz})
If $\{x_1,\ldots,x_p\}$ is an $\varepsilon$-net for $\cI$, then there is also an 
$\varepsilon$-covering of $\cI$ that consists of $p$ sets; for the $u_k$
we can take the closed balls $\cS_\varepsilon(x_k)$ with centers $x_k$ ($k=1,\ldots,p$) and radius $\varepsilon$:
$u_k=\cS_\varepsilon(x_k)\cap\cI$.
A standard theorem of topology \cite[p. 123]{Simmons} guarantees 
that each compact set $\cI$ contains a finite $\varepsilon$-net for each $\varepsilon>0$.
Hence there is also a finite $\varepsilon$-covering for each $\varepsilon>0$.
Moreover, a compact set $\cI$ can contain only finitely many $\varepsilon$-distinguishable points.
\label{rem:2}
\end{remark}

Following Kolmogorov-Tihomirov \cite{Kolmogorov} we introduce the following three functions which
characterize the \emph{massiveness} of the set $\cI$, which is supposed to be compact.

\begin{definition}
$N_\varepsilon(\cI)$ is the minimal number of sets in an $\varepsilon$-covering of $\cI$.
\label{def:4}
\end{definition}

\begin{definition}
$N_\varepsilon^Y(\cI)$ is the minimal number of points in an $\varepsilon$-net of $\cI$.
\label{def:5}
\end{definition}

\begin{definition}
$M_\varepsilon(\cI)$ is the maximal number of points in an $\varepsilon$-separated subset of $\cI$.
\label{def:6}
\end{definition}

For a given $\varepsilon>0$, the number $n$ of sets $u_k$ in a covering family depends on the family,
but the minimal value of $n$, i.e., $N_\varepsilon(\cI)\doteq\min n$ is an invariant of the set $\cI$
that depends only upon $\varepsilon$.
Similarly, the number $m$ of points in an $\varepsilon$-separated subset of $\cI$ depends on the choice
of points, but its maximum value $M_\varepsilon(\cI)\doteq\max\, m$ is an invariant of the set $\cI$ that depends 
only on $\varepsilon$.

Hereafter we shall focus on $N_\varepsilon(\cI)$ and $M_\varepsilon(\cI)$ for reasons 
which will appear clear below. Next we assign special notations for the logarithms to
the base $2$ of the functions defined above and, specifically, of the function
$N_\varepsilon(\cI)$ and $M_\varepsilon(\cI)$: 

\sku

\begin{itemize}
\item[(i)] $H_\varepsilon(\cI)\doteq\log_2 N_\varepsilon(\cI)$ is called the \emph{minimal
$\varepsilon$-entropy} of the set $\cI$, or simply the \emph{$\varepsilon$-entropy} of $\cI$.
\item[(ii)] $C_\varepsilon(\cI)\doteq\log_2 M_\varepsilon(\cI)$ is the
\emph{$\varepsilon$-capacity} of the set $\cI$.
\end{itemize}

\sku

Next we can state the following lemma.

\begin{lemma}
\label{lem:1}
For every compact set $\cI$ in the metric space $Y$ the following inequality holds:
\beq
H_\varepsilon(\cI) \leqslant C_\varepsilon(\cI).
\label{7}
\eeq
\end{lemma}

\begin{proof}
In view of the relevance of inequality \eqref{7}, we rapidly sketch, for the convenience
of the reader, the proof of this lemma, which is due (up to slight modifications) to
Kolmogorov and Tihomirov \cite{Kolmogorov}. Let $\{x_1,\ldots,x_{M_\varepsilon(\cI)}\}$ be a maximal
$\varepsilon$-separated set in $\cI$ (see Definition \ref{def:3}). It is then an 
$\varepsilon$-net of $\cI$ since in the converse case there would be a point $x'\in\cI$
such that the distance $\rho(x',x_i)>\varepsilon$; this last statement would contradict
the maximality of $\{x_1,\ldots,x_{M_\varepsilon(\cI)}\}$. In view of the fact that
$x_i\in\cI$ by definition, we obtain $N_\varepsilon^{(\cI)}(\cI)\leqslant M_\varepsilon(\cI)$.
It is evident that every $\varepsilon$-net consisting of points of $\cI$ is also
an $\varepsilon$-net consisting of points of $Y\supset\cI$, that is
$N_\varepsilon^{(Y)}(\cI)\leqslant N_\varepsilon^{(\cI)}(\cI)$. But, as we
have seen above in the remark, every $\varepsilon$-net generates an $\varepsilon$-covering
from which it follows: $N_\varepsilon(\cI)\leqslant N_\varepsilon^{(Y)}(\cI)$. 
In conclusion, we obtain: $N_\varepsilon(\cI)\leqslant M_\varepsilon(\cI)$.
Taking logarithms to the base $2$ in the last inequality, we finally obtain formula \eqref{7}.
\end{proof}

\section{$\boldsymbol{\varepsilon}$-entropy and $\boldsymbol{\varepsilon}$-capacity of compact sets}
\label{se:entropy}

From \eqref{5} we can derive the following inequality:
\beq
\left\|Af-\og\right\|_{L^2(a,b)} \leqslant \varepsilon \qquad (\varepsilon = \mathrm{constant}),
\label{8}
\eeq
where $\varepsilon$ is a bound of the noise, i.e, $\|n\|_{L^2(a,b)} \leqslant \varepsilon$. 
But, in order to truncate expansion \eqref{6} (see the Introduction) we must necessarily impose on the solution
an a-priori global bound, which can be properly written as follows
\cite{Bertero1,Bertero2,Miller}:
\beq
\|Bf\|_\cZ \leqslant E \qquad (E = \mathrm{constant}).
\label{9}
\eeq
Formula \eqref{9} amounts to requiring that there exist a positive constant $E$, a space $\cZ$
(called \emph{constraint space}) and an operator $B$ (called \emph{constraint operator}) acting as follows,
$B:X\to\cZ$ ($X=L^2(a,b)$ is the solution space) such that inequality \eqref{9} is satisfied. 
Various choice are indeed possible for the space $\cZ$ and for the constraint operator $B$, the proper 
one being mainly dependent on the type of problem under consideration. For instance, $B$ can be an appropriate
differential operator and $\cZ$ a suitably chosen Sobolev space. However, in several applications
(see the examples given in Sect. \ref{se:stability}) we can choose an operator $B$ such that
$B^*B$ commutes with $A^*A$. Furthermore, we require that the eigenvalues of the operator $B^*B$
(denoted by $\beta_k^2$) satisfy the condition $\lim_{k\to\infty}\beta_k^2=+\infty$ 
(see, as a particularly evident example, the integral equation whose kernel is given by formula \eqref{20}).
In such a case the constraint space is simply $\cZ=L^2(a,b)$, and inequality \eqref{9} reads
\beq
\sum_{k=1}^\infty \beta_k^2 \left|f_k\right|^2 \leqslant E^2 \qquad \left(f_k=(f,\psi_k)\right).
\label{10}
\eeq
At this point we must add to the assumptions (A) and (B) the following third assumption:

\skd

\noindent
{\bf Assumption (C).} The two numbers $\varepsilon$ and $E$ must be \emph{permissible}, that is,
such that the set of functions $f$ which satisfy bounds \eqref{8} and \eqref{9} is not empty.

\skd

In view of condition \eqref{10} we are led to consider the ball
$\cU\equiv\cU_{L^2(a,b)}\doteq\{f\in L^2(a,b)\,:\,\sum_{k=1}^\infty\beta_k^2\,\left|f_k\right|^2 \leqslant E^2\}$
and, accordingly, the image in the data space $Y$ of this ball under the mapping of the 
compact operator $A$ (see \eqref{1}). It is useful reminding for what follows the well-known fact
that a bounded linear compact operator $A$
from $\cH_1$ into $\cH_2$ takes bounded sets in $\cH_1$ into subsets of compact sets
in $\cH_2$; moreover, given any weakly convergent sequence $\{x_n\}$ in $\cH_1$, then the sequence
$\{A\,x_n\}$ converges strongly in $\cH_2$.

We shall consider two illustrative cases. 
In the first case we take $B=\I$ ($\I$ being the identity operator). Obviously the operator
$\I$ is self-adjoint and the eigenvalues of $\I^*\I$ are equal to $1$
(i.e., $\forall k$, $\beta_k^2=1$). Further, we assume, without loss of generality,
that the constant $E$ in formula \eqref{10} is equal to $1$ (the generalization to a constant $E\neq 1$ is straightforward).
We are thus led to consider in the constraint space $\cZ = L^2(a,b)$ the unit ball
$\cU^{(1)}\doteq\{f\in L^2(a,b)\,:\, \sum_{k=1}^\infty|f_k|^2\leqslant 1\}$ (see \eqref{10}).
The operator $A$ maps the unit ball $\cU^{(1)}$
onto a compact ellipsoid in the range of $A$ whose semi-axes are the (positive) eigenvalues $\lambda_k$
of the operator $A$, that is, the ellipsoid \cite{Prosser}
\beq
\cE^{(1)}\doteq
\left\{g\in\mathrm{Range}(A)\,:\, \sum_{k=1}^\infty \frac{|(g,\psi_k)|^2}{\lambda_k^2} \leqslant 1\right\}.
\nonumber
\eeq
As a second case we consider in $\cZ = L^2(a,b)$ the constraint operator $B$
such that $B^*B$ has eigenvalues $\beta_k^2$ with $\lim_{k\to\infty}\beta_k^2=+\infty$ 
(for simplicity, we continue to assume $E=1$).
Inequality \eqref{10} then leads to consider the ball 
$\cU^{(2)}\doteq\{f\in L^2(a,b)\,:\,\sum_{k=1}^\infty \beta_k^2\,|f_k|^2\leqslant 1\}$,
which is mapped by the operator $A$ onto the compact ellipsoid
\beq
\cE^{(2)}\doteq
\left\{g\in\mathrm{Range}(A)\,:\, \sum_{k=1}^\infty 
\left|\frac{(g,\psi_k)}{(\lambda_k/\beta_k)}\right|^2 \leqslant 1\right\}.
\nonumber
\eeq

Now, we focus our attention on the ellipsoid $\cE^{(1)}$, which is obtained by mapping the unit ball 
$\cU^{(1)}$ through the compact operator $A$.
In view of constraint \eqref{9} we consider only the solutions $f\in \cU^{(1)}$
which are mapped onto the ellipsoid $\cE^{(1)}$.
But, in general, the actual data $\og$ ($\og=g+n$) do not belong to $\cE^{(1)}$. This restriction leads us
to introduce appropriate approximations $\{\og_{\cE_1}\}$ of the data $\{\og\}$, which, for instance, can be obtained 
through suitable truncations in a way such that the approximations $\og_{\cE_1}$
belong to $\cE^{(1)}$. This point will be illustrated explicitly in what follows (see, in particular,
point (i) of Lemma \ref{lem:2}, and formulae \eqref{soluz_1} and \eqref{appr}). 
Therefore, our main problem consists in estimating the maximum
number of \emph{distinguishable messages} (this notion will be defined below) 
that can be sent back to recover an approximation of
the true (unknown) solution (assumed to belong to the unit ball $\cU^{(1)}$). 
Using the language of communication theory, 
we call \emph{distinguishable messages} those elements of $\cE^{(1)}$ which are $\varepsilon$-separated, i.e., such that
$\|\og_{\cE_1}^{(i)}-\og_{\cE_1}^{(j)}\|_Y>\varepsilon$ ($i\neq j$, $\og_{\cE_1}^{(i)},\og_{\cE_1}^{(j)}\in\cE^{(1)}$).
The ensemble composed by the maximum number of distinguishable messages that can be sent back 
to recover an approximation of the unknown solution, constitutes the backward information flow.
We first obtain an estimate of $H_\varepsilon(\cE^{(1)})=\log_2 N_\varepsilon(\cE^{(1)})$
(where $N_\varepsilon(\cE^{(1)})$ is the minimal number of sets in an $\varepsilon$-covering
of $\cE^{(1)}$; see Definition \ref{def:4}). Next, through inequality \eqref{7} we obtain a lower bound
for the maximum number of distinguishable messages that can be sent back
to recover the unknown solution, i.e., the ensemble which constitutes the backward information flow.
More precisely, the maximum number of distinguishable (or, equivalently, $\varepsilon$-separated) messages, denoted by 
$M_\varepsilon(\cE^{(1)})$ (see Definition \ref{def:6}), is larger or equal to $2^{\log_2 H_\varepsilon(\cE^{(1)})}$
(see next Corollary \ref{cor:1}).

\begin{lemma}
\label{lem:2}
Let $f\in\cU^{(1)}$, that is, assume the a-priori global bound \eqref{10}
with $\beta_k^2=1$ and $E=1$.
Then the truncation point in expansion \eqref{6}, associated with an $\varepsilon$-covering of the
ellipsoid $\cE^{(1)}$, is given by the largest
integer $k$, denoted by $k_1=k_1(\varepsilon)$, such that $\lambda_k\geqslant\varepsilon$. 
An estimate of $H_\varepsilon(\cE^{(1)})$ is given by the following inequality:
\beq
H_\varepsilon(\cE^{(1)})\geqslant\sum_{k=1}^{k_1}\log_2\frac{\lambda_k}{\varepsilon}.
\nonumber
\eeq
\end{lemma}

\begin{proof}
The image of the unit ball $\cU^{(1)}$ through the compact operator $A$ is the
ellipsoid $\cE^{(1)}$. Now, the formal series \eqref{6} is an expansion in terms of the
eigenfunctions $\{\psi_k\}_{k=1}^\infty$, which span the data space $Y$ (in this
case $Y\equiv L^2(a,b)$). Then the truncation procedure requires to consider
the intersection of the ellipsoid $\cE^{(1)}$ with the finite $k$-dimensional subspace $Y_k$ of $Y$ spanned 
by the first $k$ axes of $\cE^{(1)}$, i.e.: $\cE^{(1)}_k=\cE^{(1)} \cap Y_k$. The volume of $\cE^{(1)}_k$
is just $\prod_{n=1}^k\lambda_n$ times the volume $\Omega_k$ of the unit ball in $Y_k$. 
We now want to estimate how many balls of radius $\varepsilon$
are necessary for covering the ellipsoid $\cE^{(1)}_k$: the volume of such a ball is
$\varepsilon^k\Omega_k$, then we are forced to stop the integer $k$ at a value such that the semi-axes of the
ellipsoid $\cE^{(1)}_k$ are not smaller than the radius $\varepsilon$ of the balls. In view of the fact that 
the eigenvalues $\lambda_k$ (which coincide with the semi-axes of the ellipsoid $\cE^{(1)}$) are
a non-increasing sequence, we must take a finite subspace $Y_k$
whose dimension equals the largest integer $k$ (denoted by $k_1$) such that 
$\lambda_k\geqslant\varepsilon$. 
Now, since the volume of an $\varepsilon$-ball in $Y_{k_1}$ is given by 
$\varepsilon^{k_1}\Omega_{k_1}$, it follows that, in order to cover the ellipsoid $\cE^{(1)}$ 
by $\varepsilon$-balls, we need at least $\prod_{k=1}^{k_1}(\lambda_k/\varepsilon)$
such balls (see also Refs. \cite{Gelfand,Prosser}). In conclusion, it follows that
\begin{itemize}
\item[(i)] The truncation point of the formal expansion \eqref{6}, associated with an
$\varepsilon$-covering of the ellipsoid $\cE^{(1)}$ 
must be stopped at the largest integer $k$ (denoted by $k_1$) such that $\lambda_k\geqslant\varepsilon$.  
\item[(ii)] An estimate of the minimal number of sets in an $\varepsilon$-covering of $\cE^{(1)}$
is given by $N_\varepsilon(\cE^{(1)})\geqslant\prod_{k=1}^{k_1}(\lambda_k/\varepsilon)$ and, accordingly
\beq
H_\varepsilon(\cE^{(1)})\geqslant\sum_{k=1}^{k_1}\log_2\frac{\lambda_k}{\varepsilon}. \nonumber
\label{12}
\eeq
\end{itemize}
\end{proof}

\begin{corollary}
\label{cor:1}
Assume the conditions of Lemma \ref{lem:2}. Then the maximum number of distinguishable
messages (i.e., $\|\og^{(i)}_{\cE_1}-\og^{(j)}_{\cE_1}\|_Y>\varepsilon$ with $i\neq j$) which can be sent back from the
data set to recover the unknown function $f\in\cU^{(1)}$ is at least
\beq
M_\varepsilon(\cE^{(1)})\geqslant 2^{\left[\sum_{k=1}^{k_1}\log_2(\lambda_k/\varepsilon)\right]}. \nonumber
\label{12bis}
\eeq
\end{corollary}

\begin{proof}
The statement of the corollary is an immediate consequence of Lemma \ref{lem:2}
and of inequality \eqref{7}.
\end{proof}

\skd 

The proof of Lemma \ref{lem:2} shows that if the a-priori bound \eqref{10} holds with $\beta_k^2=1$
and $E=1$, then the maximum number of Fourier components in the 
truncated version of the formal expansion \eqref{6} is given necessarily by $k=k_1$. 
This means that, given the data $\og\in Y$,
the approximation of $f\in\cU^{(1)}$ is given by the function
\beq 
f_1(x)=\sum_{k=1}^{k_1(\varepsilon)}\frac{\og_k}{\lambda_k}\psi_k(x)
\qquad (x\in[a,b], \og_k=(\og,\psi_k)).
\label{soluz_1}
\eeq
At this point the approximation $\og_{\cE_1}$ of $\og$ can be made more transparent: we can 
write the solution as $f_1=A^{-1}\og_{\cE_1}$, the components of the approximation $\og_{\cE_1}$ being defined by
\beq
(\og_{\cE_1})_k \doteq (\og_{\cE_1},\psi_k) =
\begin{cases}
\og_k & \mathrm{for}\quad 1 \leqslant k \leqslant k_1, \\
0     & \mathrm{for}\quad k > k_1.
\end{cases}
\label{appr}
\eeq

The extension of the results of Lemma \ref{lem:2} to the case of a more general constraint operator $B$ 
is given by the following lemma.
 
\begin{lemma}
\label{lem:4}
Let $f\in\cU^{(2)}$, that is, assume the a-priori global bound \eqref{10} with $E=1$.
In addition, assume $\lim_{k\to+\infty}\beta_k^2=+\infty$.
Then:
\begin{itemize}
\item[(i)] The truncation point in expansion \eqref{6}, associated with an $\varepsilon$-covering
of the ellipsoid $\cE^{(2)}$, is given by the largest integer $k$, denoted by $k_2(\varepsilon)$, such that 
$\lambda_k\geqslant\beta_k\,\varepsilon$. Correspondingly, given the data $\og\in Y$,
the approximation of the solution $f\in\cU^{(2)}$ is given by 
\beq
f_2(x)=\sum_{k=1}^{k_2(\varepsilon)}\frac{\og_k}{\lambda_k}\psi_k(x) \qquad (x\in[a,b]). \nonumber
\label{nn}
\eeq
\item[(ii)] The maximum number of distinguishable messages (that is, the elements $\og_{\cE_2}^{(i)},\og_{\cE_2}^{(j)}\in\cE^{(2)}$
such that $\|\og_{\cE_2}^{(i)}-\og_{\cE_2}^{(j)}\|_Y>\varepsilon$, $i\neq j$) 
which can be sent back from the data set to recover the function $f\in\cU^{(2)}$ is $M_\varepsilon(\cE^{(2)})$, 
which satisfies the following inequality: 
$M_\varepsilon(\cE^{(2)})\geqslant 2^{[\sum_{k=1}^{k_2}\log_2(\lambda_k/\varepsilon)]}$.
\end{itemize}
\end{lemma}

It follows from their definition that the value of the truncation points $k_i$ ($i=1,2$) is strictly related to how
the eigenvalues $\{\lambda_k\}_{k=1}^\infty$ vary with $k$. Hille and Tamarkin
\cite{Hille} have systematically investigated the distribution of these eigenvalues on the basis of the general
regularity properties of the kernel $K(x,y)$ such as integrability, continuity, differentiability,
analyticity and the like. A very illuminating summary of all their results is given 
in Sect. 12 of \cite{Hille}. We limit ourselves to present two interesting examples, which will
be investigated in detail in Sect. \ref{se:stability}. As a first example we study the integral equation \eqref{1}
with the kernel
\beq
K(x,y) = 
\begin{cases}
(1-x)\,y & \quad\text{if \ $0\leqslant y\leqslant x\leqslant 1$}, \\
x\,(1-y) & \quad\text{if \ $0\leqslant x\leqslant y\leqslant 1$}, 
\end{cases}
\label{20}
\eeq
to which there correspond the eigenfunctions $\psi_k(x)=\sqrt{2}\sin (k\pi x)$
and the eigenvalues $\lambda_k=(k\pi)^{-2}$ ($k=1,2,\ldots$). \\
In the second example we consider the kernel
\beq
K(x,y) = \frac{\sin[c(x-y)]}{\pi(x-y)} \qquad (c=\mathrm{constant}),
\label{22}
\eeq
the asymptotic behavior for $k\to\infty$ of the corresponding eigenvalues being given by
$\lambda_k=O\{\frac{1}{k}\exp[-2k\ln(\frac{k}{ec})]\}$ \cite{Landau}.

\skd

Roughly speaking, we can say that in general the truncation point $k_i$ ($i=1,2$)
increases as the smoothness of the kernel decreases. Accordingly, the maximum
number of messages which can be sent back from the data set for recovering the solution
decreases as the smoothness of the kernel increases passing from the continuity
to the differentiability up to the analyticity.  Let us, however, remark that this criterion
must be taken with great caution: indeed, the Hille-Tamarkin results
refer essentially to the asymptotic behavior of the eigenvalues. For instance, 
in the case of the kernel given in \eqref{22}, the eigenvalues $\lambda_k$ are nearly constant
up to a certain value $k_*$ of $k$ and then decrease very rapidly for $k>k_*$.
It follows therefore that it is not admissible to keep a rigid conclusion
looking only at the asymptotic behavior of the eigenvalues.

\section{Strong and weak convergence of the truncated expansions}
\label{se:strong}

We are now in the position to study the type of convergence associated with the 
approximations $f_1(x)$ (see Lemma \ref{lem:2}) and $f_2(x)$ (see Lemma \ref{lem:4}).
First we prove the strong convergence in the $L^2$-norm of the approximation 
$f_2(x)$ and, successively, the weak convergence of the approximation $f_1(x)$
(both in the general case with $E\neq 1$).
To this end, we need the following auxiliary lemma.

\begin{lemma}
\label{lem:6}
For any function $f$ which satisfies bounds \eqref{8} and \eqref{9}, the following
inequalities hold:
\begin{subequations}
\begin{align}
& \|A(f-f_2)\|_Y\leqslant \sqrt{2}\,\varepsilon, \label{23.a} \\
& \|B(f-f_2)\|_\cZ\leqslant \sqrt{2}\,E, \label{23.b} \\
& \|A(f-f_2)\|_Y^2 + \left(\frac{\varepsilon}{E}\right)^2
\|B(f-f_2)\|_\cZ^2\leqslant 4\,\varepsilon^2, \label{23.c}
\end{align}
\label{23}
\end{subequations}
where $X\equiv Y\equiv \cZ\equiv L^2(a,b)$.
\end{lemma}

\begin{proof}
In view of bounds \eqref{8} and \eqref{9} and of the fact that the truncation point $k_2=k_2(\varepsilon,E)$ in the 
approximation $f_2(x)$ is given by the largest integer such that $\lambda_k\geqslant(\frac{\varepsilon}{E})\beta_k$, 
we have
\beq
\sum_{k=k_2+1}^\infty\lambda_k^2\,|f_k|^2 \leqslant 
\left(\frac{\varepsilon}{E}\right)^2\sum_{k=k_2+1}^\infty\beta_k^2\,|f_k|^2 \leqslant\varepsilon^2. 
\label{24}
\eeq
Therefore, taking into account inequalities \eqref{8} and \eqref{24}, we get
\beq
\begin{split}
\|A(f-f_2)\|^2_Y &= \sum_{k=1}^{k_2}\lambda_k^2\left|f_k-\frac{\og_k}{\lambda_k}\right|^2
+\sum_{k=k_2+1}^\infty\lambda_k^2\,|f_k|^2 \\
&\leqslant\|Af-\og\|^2_Y+\sum_{k=k_2+1}^\infty\lambda_k^2\,|f_k|^2\leqslant 2\varepsilon^2,
\label{25}
\end{split}
\eeq
and inequality \eqref{23.a} is proved. Analogously, in view of inequality \eqref{8}
and of the fact that $\lambda_k\geqslant(\frac{\varepsilon}{E})\beta_k$ for $k<k_2(\varepsilon,E)$, we have
\beq
\sum_{k=1}^{k_2}\beta_k^2\left|f_k-\frac{\og_k}{\lambda_k}\right|^2 \leqslant
\left(\frac{E}{\varepsilon}\right)^2 \sum_{k=1}^{k_2}\lambda_k^2\left|f_k-\frac{\og_k}{\lambda_k}\right|^2
\leqslant\left(\frac{E}{\varepsilon}\right)^2\|Af-\og\|^2_Y \leqslant E^2.
\label{26}
\eeq
Next, inequalities \eqref{9} and \eqref{26} allow us to write
\beq
\|B(f-f_2)\|^2_\cZ=
\sum_{k=1}^{k_2}\beta_k^2\left|f_k-\frac{\og_k}{\lambda_k}\right|^2 
+\sum_{k=k_2+1}^\infty\beta_k^2 \, |f_k|^2 \leqslant 2E^2, 
\label{27}
\eeq
and inequality \eqref{23.b} follows. Finally, from \eqref{25} and \eqref{27} inequality
\eqref{23.c} follows immediately.
\end{proof}

Next, we can prove the following theorem.

\begin{theorem}
\label{the:1}
Let $f$ satisfies bounds \eqref{8} and \eqref{9} and assume that the
eigenvalues $\beta_k^2$ of $B^*B$ satisfy the condition: $\lim_{k\to\infty}\beta_k^2=+\infty$.
Then the following limit holds:
\beq
\lim_{\varepsilon\to 0}\|f-f_2\|_{L^2(a,b)}=0.
\label{28}
\eeq
\end{theorem}

\begin{proof}
Let $C=A^*A+(\frac{\varepsilon}{E})^2B^*B$. The operator $C$ is evidently self-adjoint,
and we denote by $\{\gamma_k^2\}_{k=1}^\infty$ the set of its eigenvalues. Then, from inequality
\eqref{23.c} we get
\beq
\left(\sum_{k=1}^\infty\gamma_k^2\,\left|f_k-(f_2)_k\right|^2\right)^\ud\leqslant 2\varepsilon. \nonumber
\label{29}
\eeq
From this inequality it follows
\beq
\|f-f_2\|_{L^2(a,b)}=\left(\sum_{k=1}^\infty|f_k-(f_2)_k|^2\right)^{\ud}
\leqslant 
2\,\varepsilon\,\sup_k\left(\left[\lambda_k^2+\left(\frac{\varepsilon}{E}\right)^2\beta_k^2\right]^{-\ud}\right).
\label{30}
\eeq
First we note that the right-hand side of inequality \eqref{30} does not go to zero
as $\varepsilon\to 0$ if the terms $\beta_k$ are bounded. For instance, if $\beta_k=1$
we merely obtain that $\sum_{k=1}^\infty|f_k-(f_2)_k|^2\leqslant (2E)^2$.
From this observation it results evident the need to assume $\lim_{k\to\infty}\beta_k^2=+\infty$
in the statements of the theorem.
Next, we denote by $k_0=k_0(\varepsilon,E)$ the value of the integer $k$ such that, given $\varepsilon,E>0$,
$\gamma^2_{k_0}=\lambda_{k_0}^2+(\frac{\varepsilon}{E})^2\beta_{k_0}^2$
is the smallest eigenvalue of the self-adjoint operator $C$. Then, 
recalling that $\lim_{k\to\infty}\lambda_k^2=0$, it follows that 
$k_0(\varepsilon,E)\xrightarrow[\varepsilon\to 0]{}+\infty$ ($E$ fixed). We can then write
\beq
\|f-f_2\|_{L^2(a,b)}=\left(\sum_{k=1}^\infty|f_k-(f_2)_k|^2\right)^{\ud}
\leqslant 2\varepsilon E\left(\frac{1}{\varepsilon\beta_{k_0}}\right)
=2E\beta_{k_0}^{-1}\xrightarrow[\varepsilon\to 0]{} 0, \nonumber
\label{31}
\eeq
and inequality \eqref{28} follows.
\end{proof}

\begin{example}
In order to make more transparent the proof of Theorem \ref{the:1}, consider
the following example. Let $\lambda_k=k^{-\ud}$, and assume $\beta_k=k^\ud$. 
Then we consider the following simple function: $\gamma(t)=\frac{1}{t}+(\frac{\varepsilon}{E})^2 t$,
which interpolates the eigenvalues of the operator $C$ since
$\gamma(k)=\frac{1}{k}+(\frac{\varepsilon}{E})^2 k=\gamma_k^2$. 
We can now easily evaluate the minimum of $\gamma(t)$;
we have $\gamma'(t)=-\frac{1}{t^2}+(\frac{\varepsilon}{E})^2=0$, which shows that the minimum 
is attained at $t=t_0$, where 
$t_0=\left(\frac{E}{\varepsilon}\right)\xrightarrow[\varepsilon\to 0]{}+\infty$ ($E$ fixed).
\label{exa:1}
\end{example}

We can now pass to consider the truncated expansion $f_1$. First we need to prove the following lemma.

\begin{lemma}
\label{lem:7}
For any function $f$ which satisfies bound \eqref{8} and bound \eqref{9} with $B=\I$
(i.e., $\beta_k^2=1$)
the following inequalities hold:
\begin{subequations}
\begin{align}
& \|A(f-f_1)\|_Y\leqslant \sqrt{2}\,\varepsilon, \label{34.a} \\
& \|f-f_1\|_\cZ\leqslant \sqrt{2}\,E, \label{34.b} \\
& \|A(f-f_1)\|^2 _Y+ \left(\frac{\varepsilon}{E}\right)^2
\|f-f_1\|^2_\cZ \leqslant 4\,\varepsilon^2. \label{34.c}
\end{align}
\label{34}
\end{subequations}
\end{lemma}

\begin{proof}
The proof of this lemma coincides with that of Lemma \ref{lem:6}, in which we set $\beta_k^2=1$,
and recalling that the truncation point $k_1=k_1(\varepsilon,E)$ (associated with the approximation $f_1$)
is defined to be the largest $k$ such that $\lambda_k\geqslant(\frac{\varepsilon}{E})$
(see Sect. \ref{se:entropy}).
\end{proof}

\begin{theorem}
\label{the:2}
For any function $f$ which satisfies bound \eqref{8} and bound \eqref{9} with $B=\I$, the following limit holds:
\beq
\lim_{\varepsilon\to 0}\left|\left(f-f_1,v\right)\right|=0 
\qquad (v\in L^2(a,b)).
\label{35}
\eeq
\end{theorem}

\begin{proof}
Let $G=A^*A+(\frac{\varepsilon}{E})^2\I^*\I$.
Evidently $G$ is a self-adjoint operator. Inequality \eqref{34.c} can then be rewritten as follows:
\beq
(G\{f-f_1\},\{f-f_1\})\leqslant 4\varepsilon^2. \nonumber
\label{36}
\eeq 
Next, we apply the Schwarz inequality with respect to the following inner product: $[x,y]\doteq (Gx,y)$. We have
\beq
\begin{split}
(\{f-f_1\},v) &= [\{f-f_1\},G^{-1}v]\leqslant[\{f-f_1\},\{f-f_1\}]^{\ud}[G^{-1}v,G^{-1}v]^{\ud} \\
&= (G\{f-f_1\},\{f-f_1\})^{\ud}(G^{-1}v,v)^{\ud}
\leqslant 2\,\varepsilon\,(G^{-1}v,v)^{\ud} \\
&=2\,\varepsilon\left(\sum_{k=1}^\infty\frac{|v_k|^2}{\lambda_k^2+(\frac{\varepsilon}{E})^2}\right)^{\ud}.
\end{split}
\label{37}
\eeq
Now, for every $N$ we have
\beq
\sum_{k=1}^N\frac{\varepsilon^2}{\lambda_k^2+(\frac{\varepsilon}{E})^2}|v_k|^2
\leqslant E^2 \sum_{k=1}^N |v_k|^2 \leqslant E^2\,\|v\|^2<\infty. \nonumber
\label{38}
\eeq
Then, we can say that the
series $\sum_{k=1}^\infty\frac{\varepsilon^2\,|v_k|^2}{\lambda_k^2+(\varepsilon/E)^2}$ converges uniformly.
Since for any $k$ we have $\lim_{\varepsilon\to 0}\frac{\varepsilon^2\,|v_k|^2}{\lambda_k^2+(\varepsilon/E)^2}=0$,
then
\beq
\lim_{\varepsilon\to 0}2\varepsilon\left(\sum_{k=1}^\infty
\frac{|v_k|^2}{\lambda_k^2+(\frac{\varepsilon}{E})^2}\right)^\ud=0, \nonumber
\label{39}
\eeq
and therefore from inequality \eqref{37} statement \eqref{35} follows.
\end{proof}

\skd

\section{Stability estimates: H\"{o}lder and logarithmic continuity}
\label{se:stability}

\subsection{A-priori information and the backward information flow}
\label{subse:apriori}

Our goal now is to obtain an upper bound on the \emph{approximation error} $\|f-f_2\|_{L^2(a,b)}$
associated with the truncated approximation $f_2$ that, in the previous section, 
we have shown to converge strongly to the unknown function $f$ for $\varepsilon\to 0$.
For this purpose, we introduce the \emph{stability estimate}
$\fM(\varepsilon,E)$, which is defined as follows \cite{Miller} (hereafter we simply denote
by $\|\cdot\|$ the norm in $L^2(a,b)$):
\beq
\fM(\varepsilon,E)\doteq\sup\{\|f\|,f\in X\equiv L^2(a,b)\,:\,\|Af\|\leqslant\varepsilon,\|Bf\|\leqslant E\}.
\label{42}
\eeq
The quantity $\fM(\varepsilon,E)$ gives the size, in the sense of the norm $\|\cdot\|$, of the packet
of functions satisfying conditions \eqref{8} and \eqref{9}. In fact, if there exist two approximations
of $f$, and both of them satisfy conditions \eqref{8} and \eqref{9}, then it is easy to see that the 
$L^2$-norm of their difference is bounded by $2 \,\fM(\varepsilon,E)$. If $\fM(\varepsilon,E)$
tends to zero as $\varepsilon\to 0$ ($E$ fixed), the size of the packet collapses and we can
thus say that the problem of finding an approximation of $f$ (which satisfies conditions \eqref{8} and \eqref{9})
is stable with respect to the norm $\|\cdot\|$. We can therefore appropriately call $\fM(\varepsilon,E)$
the \emph{stability estimate} and an upper bound of $\fM(\varepsilon,E)$ the
\emph{best possible stability estimate}. Next, if we consider the approximation $f_2(x)$
(which converges strongly to $f(x)$) and take into account inequalities \eqref{23.a} and \eqref{23.b}
and definition \eqref{42}, we obtain the following bound on the \emph{approximation error}:
$\|f-f_2\|\leqslant \sqrt{2}\,\fM(\varepsilon,E)$. Furthermore, it is rather interesting to investigate how fast
$\fM(\varepsilon,E)$ tends to zero as $\varepsilon\to 0$ ($E$ fixed). We can indeed distinguish
between H\"older-type and logarithmic-type dependence of $\fM(\varepsilon,E)$ on $\varepsilon$
($E$ fixed). To this end, it is convenient to state the following lemma.

\begin{lemma}
\label{lem:8}
Assume that the eigenvalues $\{\lambda_k^2\}_{k=1}^\infty$ associated with the operator $A^*A$ 
(see \eqref{1}) and the eigenvalues $\{\beta_k^2\}_{k=1}^\infty$ associated with the operator $B^*B$
satisfy the following inequality: 
\beq
\lambda_k^2\geqslant\beta_k^2\ p(\beta_k^{-2}),
\label{43prev}
\eeq
where the function $r\mapsto p(r)$ enjoys the following properties:
\begin{itemize}\addtolength{\itemsep}{0.1\baselineskip}
\item[(i)] $0<r\mapsto\frac{p(r)}{r}$ is a positive increasing function;
\item[(ii)] $p(0^+)=0$;
\item[(iii)] $p(r)$ is a convex function.
\end{itemize}
Then the following inequality holds:
\beq
\fM(\varepsilon,E)\leqslant E \cdot \sqrt{p^{-1}\!\!\left(\frac{\varepsilon^2}{E^2}\right)},
\label{43}
\eeq
where $p^{-1}$ denotes the inverse of $p$.
\end{lemma}

\begin{proof}
We start from Jensen's inequality, which holds in view of the convexity of the function $p$
\beq
p\!\left(\sum_{k=1}^\infty a_k b_k\right) \leqslant \sum_{k=1}^\infty a_k \ p(b_k)
\qquad \left(\sum_{k=1}^\infty a_k=1,\, a_k \geqslant 0\right).
\label{44}
\eeq
Next, we put in \eqref{44}: $b_k=\beta_k^{-2}$ and
$a_k=\frac{\beta_k^2 |f_k|^2}{\|Bf\|^2}$, so that $\sum_{k=1}^\infty a_k=1$.
Then, from \eqref{44} and using assumption \eqref{43prev}, we have
\beq
\begin{split}
p\!\left(\sum_{k=1}^\infty\frac{\beta_k^2 |f_k|^2}{\|Bf\|^2}\frac{1}{\beta_k^2}\right)
&=p\!\left(\frac{\|f\|^2}{\|Bf\|^2}\right)
\leqslant\sum_{k=1}^\infty\frac{\beta_k^2 |f_k|^2}{\|Bf\|^2}\,p\!\left(\frac{1}{\beta_k^2}\right) \\
&\leqslant\sum_{k=1}^\infty\frac{\lambda_k^2 |f_k|^2}{\|Bf\|^2}
=\frac{\|Af\|^2}{\|Bf\|^2},
\end{split} \nonumber
\label{45}
\eeq
from which we extract the inequality we need
\beq
p\!\left(\frac{\|f\|^2}{\|Bf\|^2}\right)\leqslant\frac{\|Af\|^2}{\|Bf\|^2}.
\label{45b}
\eeq
Now, we set
$r_1=\|f\|^2/E^2$ and $r_2=\|f\|^2/\|Bf\|^2$, and we have $r_1\leqslant r_2$ for $\|Bf\|\leqslant E$.
The monotonicity assumption (i) for $p(r)/r$ yields the following inequality:
\beq
E^2 \ p\!\left(\frac{\|f\|^2}{E^2}\right) \leqslant \|Bf\|^2 \ p\!\left(\frac{\|f\|^2}{\|Bf\|^2}\right).
\label{45c}
\eeq
Finally, combining \eqref{45b} and \eqref{45c}, and recalling definition \eqref{42} of $\fM(\varepsilon,E)$,
inequality \eqref{43} follows. 
\end{proof}

\sku

\begin{example}
\label{exa:2}
As a first example consider the kernel \eqref{20} and the corresponding
eigenvalues $\lambda_k=(k\pi)^{-2}$. It is worth observing that the assumption
$\lambda_k^2\geqslant\beta_k^2\ p(\beta_k^{-2})$ (along with the assumed properties for $p(r)$) 
implies necessarily that $\lim_{k\to\infty}\beta_k^2=+\infty$, as required by
Theorem \ref{the:1} in order to have the strong convergence of the approximation $f_2$.
Now, we assume the function $p(r)$ to be: $p(r)=r^{1/\gamma}$
($0<\gamma<1$), which satisfies conditions (i), (ii), and (iii) 
of Lemma \ref{lem:8}. Then inequality \eqref{43prev} reads:
$\lambda_k^2\geqslant\beta_k^{[2(\gamma-1)/\gamma]}$, which leads to the condition the eigenvalues
$\beta_k$ are required to satisfy, i.e., $\beta_k\geqslant(k\pi)^{[2\gamma/(1-\gamma)]}$.
Note that, since $\frac{2\gamma}{(1-\gamma)}>0$, the latter inequality implies that the sequence of eigenvalues
$\{\beta_k^2\}_{k=1}^\infty$ must tend (sufficiently fast) to infinity for $k\to\infty$.
For this example it is easy to find a differential operator $B$ such that $B^*B$ commute with $A^*A$
and whose eigenvalues $\beta_k^2$ satisfy the condition $\lim_{k\to\infty}\beta_k^2=+\infty$.
Let us set $\gamma=\frac{1}{3}$; accordingly, from the inequality 
$\beta_k\geqslant(k\pi)^{[2\gamma/(1-\gamma)]}$ we obtain: $\beta_k\geqslant k\pi$. Therefore 
we can take as constraint operator the first derivative $B=\frac{\rmd}{\rmd x}$, which corresponds to $\beta_k=k\pi$.
Now, for the sake of simplicity we put $E=1$; then, the truncation point $k_2(\varepsilon)$ associated 
with the approximation $f_2(x)$ ($k_2(\varepsilon)$ being the largest value of $k$ 
such that $\lambda_k\geqslant\varepsilon\beta_k$) is given by:
$k_2(\varepsilon)=(\pi\sqrt[3]{\varepsilon})^{-1}$, 
which is smaller than 
$k_1(\varepsilon)=(\pi\sqrt{\varepsilon})^{-1}$ 
($\varepsilon\ll 1$), representing the truncation point associated with the approximation $f_1(x)$ which
converges weakly to $f$ (see Theorem \ref{the:2}). On the other hand, from formula 
\eqref{43} we obtain an upper bound for the stability estimate, which is given by
$\fM(\varepsilon,1)\leqslant\sqrt[3]{\varepsilon}$. 
Accordingly, the \emph{approximation error} is bounded as follows: $\|f-f_2\|\leqslant \sqrt{2}\sqrt[3]{\varepsilon}$, 
which tends rapidly to zero as $\varepsilon\to 0$: we have a H\"older-type continuity in the stability estimate.
\end{example}

Now, we have reached an apparently paradoxical situation: if we consider the approximation
$f_1(x)$, which converges only weakly to $f(x)$ (see Theorem \ref{the:2}), the maximum
number of messages which can be conveyed back from the data set for recovering the solution 
is given at least by $2^{[\sum_{k=1}^{k_1(\varepsilon)}\log_2\frac{\lambda_k}{\varepsilon}]}$,
which is larger than $2^{[\sum_{k=1}^{k_2(\varepsilon)}\log_2\frac{\lambda_k}{\varepsilon}]}$,
which represents the maximum number of messages that can be sent back from the data set
associated with the approximation $f_2(x)$, converging strongly to $f$ (see Theorem \ref{the:1}).

The paradox outlined above goes beyond the specific example illustrated so far and it is quite general.
It can be solved distinguishing between \emph{a-priori information} and
\emph{transmitted information}. 
The greater the amount of \emph{a-priori information} which restricts the class of the admissible 
solutions (that is, the stronger the a-priori bound $\|Bf\|_\cZ\leqslant E$), 
the smaller the amount of \emph{information transmitted back} (that is, the smaller the maximum number 
of messages which should be conveyed back to reconstruct the solution).

\skd

Let us now consider the second example of Sect. \ref{se:entropy} (see formula \eqref{22}).
The eigenfunctions of the integral operator $A$ (see \eqref{1}) acting as follows
$A: L^2(-1,1)\to L^2(-1,1)$, whose kernel is $K(x,y)=\frac{\sin[c(x-y)]}{\pi(x-y)}$
(see \eqref{22}), are the so-called \emph{prolate spheroidal functions} and are denoted
by $\psi_k(c,x)$ \cite{Flammer,Frieden,Landau,Slepian}.
They can be defined as the continuous solutions, on the closed interval $[-1,1]$,
of the following differential equation \cite{Flammer}:
\beq
-[(1-x^2)\psi'(x)]'+c^2\psi^2(x)=\chi\,\psi(x).
\label{prolate}
\eeq
Continuous solutions exist only for certain discrete positive values $\chi_k$ of the parameter $\chi$:
$0<\chi_0<\chi_1<\cdots$. Then $\psi_k(c,x)$ is just the solution of \eqref{prolate} corresponding
to the eigenvalues $\chi_k$. The behavior of $\chi_k$ when $k\to+\infty$ is \cite{Flammer}
\beq
\chi_k=k(k+1)+\uds c^2 + O\left(k^{-2}\right).
\nonumber
\label{prolate1}
\eeq
Then, through the differential operator given on the left-hand side of \eqref{prolate} we can obtain
the operator $B^*B$ that commutes with $A^*A$. But let us note that while the eigenvalues of $B^*B$
present a power-like increase for $k\to+\infty$, the eigenvalues $\lambda_k$ decrease for $k\to+\infty$
as $\lambda_k=O\left\{\frac{1}{k}\exp[-2k\ln(\frac{k}{ec})]\right\}$ \cite{Landau}.
In this case the results of Lemma \ref{lem:8} cannot be applied in a strict sense 
but we should limit ourselves to considerations holding
only for sufficiently large values of $k$. We note that if we choose a function $p(r)$ such that
$p(r)\staccrel{\sim}{r\to 0}4r\exp(-2/r)$ then 
$p^{-1}(s)\staccrel{\sim}{s\to 0}2\left|\ln\frac{s}{4}\right|^{-1}$ \cite{Bertero2,Talenti}.
With this choice of $p(r)$ inequality \eqref{43prev} gives a condition leading to 
a sequence $\{\beta_k^2\}_{k=1}^\infty$ which presents a divergence of the following type:
$\beta_k^2\gtrsim 2k\ln(k/ec)$. Correspondingly, from \eqref{43}
we can argue that the stability estimate $\fM(\varepsilon,1)$ (we put $E=1$) for the approximation
$f_2(x)$ satisfies (in a neighborhood of $\varepsilon=0^+$) a logarithmic-type continuity, i.e.,
$\fM(\varepsilon,1)\lesssim |\ln(\varepsilon/2)|^{-1/2}$.

For this latter example it is interesting to study also the weakly convergent approximation $f_1(x)$. To this end,
we rewrite the integral equation \eqref{1} with kernel \eqref{22} (using now, for later convenience, the standard notation of
optics and information theory, i.e., $\Omega=c$ and support of the functions
$(-\frac{X}{2},\frac{X}{2})$ instead of $(-1,1)$)
\beq
(Af)(x) = \int_{-X/2}^{X/2}\frac{\sin[\Omega(x-y)]}{\pi(x-y)}f(y)\,\rmd y = g(x)
\quad
\left(f,g\in L^2\left(-\frac{X}{2},\frac{X}{2}\right)\right), \nonumber
\label{47}
\eeq
which can be written as follows:
\beq
\frac{1}{2\pi}\int_{-\Omega}^{\Omega}\rmd\omega\,e^{-\rmi\omega x}
\left[\int_{-X/2}^{X/2}e^{\rmi\omega y}f(y)\,\rmd y\right]=g(x).
\label{48}
\eeq
Next, the left-hand side of \eqref{48} can be split further by writing
\begin{align}
& g(x)=\frac{1}{\sqrt{2\pi}}\int_{-\Omega}^{\Omega} F(\omega) e^{-\rmi\omega x}\,\rmd\omega, \label{49} \\
\intertext{where}
& F(\omega) = \frac{1}{\sqrt{2\pi}}\int_{-X/2}^{X/2}e^{\rmi\omega y}f(y)\,\rmd y. \label{50}
\end{align}
Equality \eqref{50} shows that $F(\omega)$ is an entire function in the $\omega$-plane since
$f(y)$ has compact support, vanishing outside the interval of length $X$. But also equality
\eqref{49} can be regarded as the Fourier transform of a function which is given by
$F(\omega)$ in the interval $-\Omega\leqslant\omega\leqslant\Omega$, and zero outside.
Therefore also $g(x)$ is an entire function in the $x$-plane and then can be reconstructed
by a discrete collection of its values, chosen in arithmetic progression with difference
$\frac{\pi}{\Omega}$, as proved originally by De La Vall\'ee-Poussin.
In particular, the function $g(x)$ can be reconstructed in the interval
$\left(-\frac{X}{2},\frac{X}{2}\right)$ of length $X$ from the knowledge of its values 
on a set of $S$ points, where $S\doteq\frac{X}{\pi/\Omega}=\frac{\Omega X}{\pi}$ is usually 
called in information theory and in optics the \emph{Shannon number} \cite{Toraldo}. 
Now, it turns out that if $S$ is sufficiently large the eigenvalues $\{\lambda_k\}_{k=1}^\infty$ form a non-increasing
ordered sequence (i.e., $1\geqslant\lambda_1\geqslant\lambda_2\geqslant\ldots$)
which enjoys a step-like behavior \cite{Frieden,Slepian,Toraldo}: they are approximately equal to $1$
for $k<\lfloor S\rfloor+1$ and, successively, for $k\geqslant\lfloor S\rfloor+1$ fall off to zero
very rapidly (the symbol $\lfloor x\rfloor$ standing for the integral part of $x$). 
Therefore, if we return to the approximation $f_1(x)$ and, accordingly, to Lemma \ref{lem:2}, we obtain
the following estimate for the $\varepsilon$-entropy $H_\varepsilon(\cE^{(1)})$:
\beq
\sum_{k=1}^{k_1}\log_2\frac{\lambda_k}{\varepsilon}
=\sum_{k=1}^{\lfloor S\rfloor}\log_2\frac{\lambda_k}{\varepsilon}
+\sum_{k=\lfloor S\rfloor+1}^{k_1}\log_2\frac{\lambda_k}{\varepsilon}.
\label{51}
\eeq 
Since for $k\leqslant\lfloor S\rfloor$ we have $\lambda_k\simeq 1$, the contribution
of the first sum on the right-hand side of formula \eqref{51} is approximately given by
$S\log_2\varepsilon^{-1}$; moreover, for $k\geqslant\lfloor S\rfloor+1$,
the eigenvalues are approximately given by $\lambda_k\simeq\varepsilon$, then
the second sum on the right-hand side of \eqref{51} is nearly null. We can thus conclude that the
maximum number of messages sent back from the data set is at least given by
\beq
M_\varepsilon(\cE^{(1)}) \gtrsim 2^{\left(S\log_2\varepsilon^{-1}\right)}
\xrightarrow[\varepsilon\to 0]{} +\infty, \nonumber
\label{52}
\eeq
which gives a simple and clear estimate of the backward information flow.

\section{Conclusions}
\label{se:conclusions}

We can now draw the following conclusions.
\begin{itemize}
\item[1.] The regularization of the Fredholm integral equations of the first kind,
realized by truncating the expansions in terms of eigenfunctions of the integral
operator, can be reconsidered from the viewpoint of the $\varepsilon$-covering of compacta.
More specifically, the truncation points of the expansions can be determined by studying the minimal number
of sets in an $\varepsilon$-covering of compact ellipsoids. From these evaluations
we can recover an estimate of the maximum number of messages which can be conveyed back
from the data set to recover the solution.
\item[2.] We obtain two different classes of approximations: strongly convergent
approximations and weakly convergent approximations.
\item[3.] In the case of strongly convergent approximation it is interesting to
control the \emph{stability estimate} and, accordingly, the \emph{reconstruction error}.
We can distinguish between a H\"older-type stability and a logarithmic-type stability.
In the first case the \emph{error function} presents a H\"older-type dependence
on the noise $\varepsilon$; in the second case the \emph{error function} depends
logarithmically on the noise $\varepsilon$, i.e., $\sim|\ln(\frac{\varepsilon}{E})|^{-1/2}$.  
\item[4.] Regarding the strongly convergent approximation, we encounter an apparently paradoxical situation:
the maximum number of messages which can be conveyed back from the data set for recovering 
the solution is smaller if the \emph{a-priori bound} is stronger and, accordingly,
if the class of admissible solutions is stricter.
The paradox can be solved by distinguishing between \emph{a-priori information}
and \emph{transmitted information}. If the \emph{a-priori bound}, which limits the class
of the admissible solutions, is very strict it follows that a small number of messages, sent back from
the data set, is sufficient to recover the solution.
To a greater amount of \emph{a-priori information} there corresponds a smaller amount of
\emph{transmitted information} which is necessary for finding the unknown solution.
\item[5.] Point (4) sheds light on the relevance of Assumption (A)
(see the Introduction). The standard regularization procedures (specifically, in the case
of the truncated approximations) work only if it is possible to introduce
\emph{a-priori bounds} such that the components $\og_k$ ($\og_k=(\og,\psi_k)$)
which have been retained in the approximate solution are those carrying the bulk of the
unknown solution, while those which are cut off can be actually neglected.
\end{itemize}

\skt

\end{document}